\documentclass[12pt,reqno]{amsart}

\usepackage{amssymb}
\usepackage{amsmath}
\usepackage{amsthm}
\usepackage{enumerate}
\usepackage{graphicx}
\usepackage{color}
\usepackage{mathrsfs}
\usepackage[hidelinks,hypertexnames=false]{hyperref}

\theoremstyle{plain}
\newtheorem{thm}{Theorem}[section]
\newtheorem*{thmA}{Theorem A}
\newtheorem{lem}[thm]{Lemma}

\newtheorem{de}{Definition}[section]
\theoremstyle{remark}
\newtheorem{re}{Remark}[section]
\theoremstyle{plain}
\newtheorem{question}{Question}[section]

\numberwithin{equation}{section}

\begin{document}

\title[Prescribed Beurling densities in spectral eigen-subspaces]
{Prescribed Beurling densities in spectral eigen-subspaces}

\author{Zheng-Yi Lu}

\address{School of Mathematics, Hunan University, Changsha Hunan 410082, China}

\email{zyluhnsd@163.com, zyluzky@amss.ac.cn}

\thanks{The research is supported in part by the NNSF of China (Nos. 12501111, 11831007 and 12071125), the
Hunan Provincial NSF (No. 2025JJ60065), the Fundamental Research Funds for the Central Universities.}
\begin{abstract}
Let $(q,D,L)$ be a one-dimensional Hadamard triple with
$2\leq m:=\#D<q$, and let $\mu_{q,D}$ be the associated self-similar
probability measure. Write
$\mathcal T=\{t\in\mathbb Z:(q,D,tL)\text{ is a Hadamard triple}\}$.
We prove that, for every $0<s<\log m/\log q$ and every
$\zeta\in[0,+\infty]$, there are continuum many spectra $\Lambda$ of
$\mu_{q,D}$ such that $t\Lambda$ is a spectrum for every $t\in\mathcal T$,
$\dim_{\mathrm{Be}}\Lambda=s$, and $D_s^+(\Lambda)=\zeta$.
The construction combines a nested family of finite digit sets with
lacunary spectral extensions. A two-sided interval-counting estimate
controls the density of the prescribed subset, while the added frequencies
ensure completeness for every multiplier without changing that density.
Binary choices in the extension yield continuum many distinct spectra.
This refines Beurling dimension interpolation in spectral eigen-subspaces
by prescribing the upper density at every strictly intermediate dimension.
\end{abstract}
\maketitle

\section{Introduction}
The classical Fourier series on an interval expresses a square-integrable
function as a superposition of mutually orthogonal exponentials.
The spectral set problem asks which other domains admit such an
expansion and how the geometry of a domain is reflected in its
frequencies. A measurable set $K\subset\mathbb R^d$ with
$0<|K|<\infty$ is called a \emph{spectral set} if there exists a
countable set $\Lambda\subset\mathbb R^d$ such that
\[
 \{|K|^{-1/2}e^{2\pi i\langle\lambda,x\rangle}:\lambda\in\Lambda\}
\]
is an orthonormal basis of $L^2(K)$. It is a \emph{translational tile}
if there is a discrete set $\mathcal A\subset\mathbb R^d$ satisfying
\[
                   \sum_{a\in\mathcal A}\mathbf1_K(x-a)=1
                       \quad\text{for almost every }x\in\mathbb R^d.
\]
Thus the translates cover the ambient space without overlaps of
positive measure. Fuglede's conjecture \cite{Fug74}, proposed in
1974, asserts that spectral sets and translational tiles are exactly
the same sets. For a fundamental domain of a lattice, the dual
lattice provides the familiar exponential basis, illustrating the
connection between the two properties.

The conjecture stimulated extensive interaction among Fourier
analysis, discrete geometry, number theory, and finite-group
combinatorics. Tao \cite{Tao04} constructed spectral sets that do not
tile in dimensions at least five. Kolountzakis and Matolcsi
\cite{KM06} constructed tiles without spectra; further
counterexamples in dimension four were obtained in
\cite{FR06,Ma05}. These developments show that neither implication
holds for unrestricted sets in general. Geometric restrictions,
however, can recover a precise correspondence. For example,
{\L}aba \cite{La01} established the conjecture for unions of two
intervals, and Lev and Matolcsi \cite{LM22} proved it for convex
bodies in every dimension. These results highlight the role of geometric and arithmetic structure
in the relation between spectrality and tiling.

The measure-theoretic formulation extends this question beyond
sets of positive Lebesgue measure. A Borel probability measure
$\mu$ on $\mathbb R$ is called a \emph{spectral measure} if there is
a countable set $\Lambda\subset\mathbb R$ such that
\[
 E(\Lambda)=\{e_\lambda(x)=e^{2\pi i\lambda x}:\lambda\in\Lambda\}
\]
is an orthonormal basis for $L^2(\mu)$. Such a set $\Lambda$ is called
a \emph{spectrum} of $\mu$. Normalized Lebesgue measure on a spectral
set fits this definition. A singular measure, in contrast, is
concentrated on a set of Lebesgue measure zero, so the existence of
an exponential basis cannot be explained by the usual volume-based
tiling picture.

Jorgensen and Pedersen \cite{JP98} showed that this distinction does
not prevent Fourier expansions. The self-similar measure with
contraction $1/4$ and digits $\{0,2\}$ is spectral, with spectrum
\[
 \left\{\sum_{j=0}^{n-1}\varepsilon_j4^j:
              n\geq1,\ \varepsilon_j\in\{0,1\}\right\},
\]
whereas the standard middle-third Cantor measure is not spectral.
These examples demonstrate that self-similarity alone does not
determine spectrality: the arithmetic relation between the
contraction and the digit set is essential. They led to systematic
studies of self-similar, self-affine, and infinite-convolution
measures, including the spectrality results in
\cite{AFL19,AL23,Dai12,DC21,DHL19,LW02,LMW22,LLZ23}. 

Singular spectral measures are therefore natural objects for
studying Fourier analysis on fractal supports. An exponential basis
gives coefficients, Parseval identities, and $L^2$ expansions for
functions on such supports, while questions about convergence of
the corresponding Fourier series lead to additional analytic
problems; see \cite{DHS14,Str00,Str06}. At the same time, their
spectra connect the small-scale structure of a measure with the
large-scale distribution of a discrete frequency set. In the
Lebesgue setting, Landau's density theory \cite{Lan67} ties sampling
and interpolation to the volume of the underlying domain.
For singular measures, Beurling dimension describes the polynomial
growth of frequencies; see the study of exponential frames in
\cite{DHSW11}.

For a fixed spectral measure, one may further ask which growth
exponents and densities its spectra can attain, and which dilations
preserve spectrality. We study these questions simultaneously for
singular self-similar measures: we prescribe both the Beurling dimension
and the upper density while requiring a family of integer dilates
to remain spectra.

We consider self-similar probability measures on the line. Let
$q\geq2$ be an integer and let $D\subset\mathbb Z$ be finite. The
associated measure is the unique probability measure satisfying
\begin{equation}\tag{1.1}\label{(1.1)}
 \mu_{q,D}(A)=\frac1{\#D}\sum_{d\in D}\mu_{q,D}(qA-d)
 \qquad(A\subset\mathbb R\text{ Borel});
\end{equation}
see \cite{Hut81}. Equivalently,
$\mu_{q,D}=\mathop{*}_{j=1}^{\infty}\delta_{q^{-j}D}$, where
$\delta_F=(\#F)^{-1}\sum_{x\in F}\delta_x$ for a finite nonempty set $F$.
A useful algebraic condition for spectrality is the following.

\begin{de}\label{de1.1}
Let $q\geq2$ be an integer and let $D,L\subset\mathbb Z$ satisfy
$\#D=\#L=m$. The triple $(q,D,L)$ is a \emph{Hadamard triple} if
\[
 \frac1{\sqrt m}\left(e^{2\pi i d\ell/q}\right)_{d\in D,\,\ell\in L}
\]
is unitary.
\end{de}

For one-dimensional Hadamard triples, spectrality of the associated
self-similar measures was established in \cite{LW02}. The
higher-dimensional Hadamard triple theorem of Dutkay, Haussermann,
and Lai \cite{DHL19} establishes spectrality for the corresponding
self-affine measures. Throughout this
paper we assume
\[
                  2\leq m:=\#D=\#L<q,
 \qquad r_0:=\frac{\log m}{\log q}.
\]
The existence of a spectrum leaves considerable freedom in the
choice of its frequencies. One aspect of this freedom concerns
\emph{spectral eigenvalues}: numbers $t$ for which a spectrum and its
dilate $t\Lambda$ are both spectra of the same measure. This problem
has been studied in \cite{Dai16,DH16,FHW18,HTW19,KLW25}.
For a set $\mathcal A$ of nonzero real numbers containing $1$, write
\[
 V(\mu,\mathcal A)=\{\Lambda:\ a\Lambda\text{ is a spectrum of }\mu
                         \text{ for every }a\in\mathcal A\}.
\]
We call this collection a \emph{spectral eigen-subspace}; the term
does not imply a vector-space structure. The distinct spectra constructed
below all have minimum zero and are therefore inequivalent under translation.

The multiplier set relevant to a Hadamard triple is
\begin{equation}\tag{1.2}\label{(1.2)}
 \mathcal T=\{t\in\mathbb Z:(q,D,tL)\text{ is a Hadamard triple}\}.
\end{equation}
Here $1\in\mathcal T$, $0\notin\mathcal T$, and
$\mathcal T=-\mathcal T$. Deng, Fu and Kang \cite{DFK24} studied common
spectra for the odd integer multipliers of the quarter Cantor measure.
The results of \cite{Lu26} establish common spectra and Beurling
dimension interpolation for the multiplier set \eqref{(1.2)}.

To describe the distribution of frequencies more precisely, put
$B(x,h)=(x-h,x+h)$ and define, for $s>0$,
\[
 M_\Lambda(h)=\sup_{x\in\mathbb R}\#(\Lambda\cap B(x,h)),
 \qquad
 D_s^+(\Lambda)=\limsup_{h\to\infty}\frac{M_\Lambda(h)}{h^s}.
\]
The denominator is the $s$th power of the \emph{radius}; this convention
is used consistently throughout. The Beurling dimension is
\[
 \dim_{\mathrm{Be}}\Lambda
 =\inf\{u>0:D_u^+(\Lambda)=0\}
 =\sup\{u>0:D_u^+(\Lambda)>0\},
\]
with the usual convention for dimension zero. Beurling dimension
quantifies polynomial growth, whereas density at that dimension
retains information that the exponent alone does not record.
Dimension estimates for spectra of self-similar measures were studied
in \cite{DHSW11,HKTW18}; see also the entropy bound in \cite{Shi21}.
Constructions of sparse spectra and dimension interpolation appear
in \cite{AL231,LW23,LW24}.

\begin{thmA}[\cite{Lu26}]
Let $(q,D,L)$ be a one-dimensional Hadamard triple with $2\leq m<q$.
For every $s\in[0,r_0]$, the collection
\[
 V^{(s)}(\mu_{q,D},\mathcal T)
 =\{\Lambda\in V(\mu_{q,D},\mathcal T):
                       \dim_{\mathrm{Be}}\Lambda=s\}
\]
has cardinality $2^{\aleph_0}$.
\end{thmA}

This result prescribes the growth exponent. It does not prescribe the
upper density at that exponent: a set of dimension $s$ may have zero,
finite positive, or infinite upper $s$-Beurling density. We therefore
ask the following finer question.

\begin{question}\label{qu1}
For $0<s<r_0$ and $\zeta\in[0,+\infty]$, does
\begin{equation}\tag{1.3}\label{(1.3)}
 V^{(s)}_\zeta(\mu_{q,D},\mathcal T)
 =\{\Lambda\in V(\mu_{q,D},\mathcal T):
       \dim_{\mathrm{Be}}\Lambda=s,\ D_s^+(\Lambda)=\zeta\}
\end{equation}
have cardinality $2^{\aleph_0}$?
\end{question}

Our main result answers Question \ref{qu1} affirmatively.
\begin{thm}\label{thm1.1}
Let $(q,D,L)$ be a one-dimensional Hadamard triple with
$2\leq m:=\#D<q$. Let $\mu_{q,D}$, $\mathcal T$, and
$V^{(s)}_\zeta(\mu_{q,D},\mathcal T)$ be defined by
\eqref{(1.1)}, \eqref{(1.2)}, and \eqref{(1.3)}, respectively.
For every $0<s<\log m/\log q$ and $\zeta\in[0,+\infty]$,
\[
                 \#V^{(s)}_\zeta(\mu_{q,D},\mathcal T)=2^{\aleph_0}.
\]
The spectra may be chosen in $\mathbb Z_{\geq0}$ with minimum zero.
\end{thm}

The condition $m<q$ cannot be omitted from the theorem as stated.
For $(q,D,L)=(2,\{0,1\},\{0,1\})$, the measure is Lebesgue measure
on $[0,1]$. After translating a spectrum to contain zero, its
orthogonality forces it to be a subset of $\mathbb Z$. Completeness
then forces it to equal $\mathbb Z$, so every spectrum has Beurling
dimension one. Thus the asserted interpolation for $0<s<1$ fails
in this full-digit example.

The proof has two components. First, we construct an increasing
family of digit words whose numerical images have the prescribed
density. Long blocks of zero digits separate consecutive stages; a block of
selected digit words is followed by a block of unrestricted digits.
For each resulting finite set, its interval counts differ from a
fixed proportion of the full digit set by at most twice the number
of selected prefixes. This two-sided estimate permits exact density
calibration. Second, we preserve these frequencies while adding
lacunary frequencies to obtain completeness for every multiplier.
The added frequencies contribute only $O(\log h)$ points to any
interval of radius $h$, so they do not change the upper $u$-Beurling density for any $u>0$. Independent choices among them produce the required
cardinality. The endpoint cases $\zeta=0$ and $\zeta=+\infty$ are
handled with slowly varying positive target densities.

Section \ref{sect.2} gives the Fourier tools. Section \ref{sect.3}
establishes the recursive spectral extension. Section \ref{sect4}
constructs the prescribed-density digit sets and proves
Theorem \ref{thm1.1}.

\section{Preliminaries}\label{sect.2}
We recall the finite Fourier identities and completeness criteria
used in the construction. We state the cited results without proof and then derive the
selection estimates needed for the recursive construction.
We use the convention
$\widehat\mu(\xi)=\int e^{-2\pi i\xi x}\,d\mu(x)$.
Write $\mu=\mu_{q,D}$ and
\[
 \mu_n=\mathop{*}_{j=1}^n\delta_{q^{-j}D},
 \qquad m_D(\xi)=\frac1m\sum_{d\in D}e^{-2\pi i d\xi}.
\]
Then
\begin{equation}\tag{2.1}\label{(2.1)}
 \widehat\mu(\xi)=\prod_{j=1}^{\infty}m_D(q^{-j}\xi)
 =\widehat{\mu_n}(\xi)\widehat\mu(q^{-n}\xi).
\end{equation}
In particular, $\widehat{\mu_n}$ is $q^n$-periodic. A set $\Lambda$
is orthogonal for $\mu$ if and only if
\begin{equation}\tag{2.2}\label{(2.2)}
 (\Lambda-\Lambda)\setminus\{0\}
               \subset\mathcal Z(\widehat\mu),
 \qquad \mathcal Z(f)=\{x:f(x)=0\}.
\end{equation}

The following standard consequences of Hadamard unitarity connect
the finite Fourier matrix with the spectra of the partial convolutions.
\begin{lem}[\cite{LW02}]\label{thm2.1}
Let $q\geq2$ and $D,L\subset\mathbb Z$ with $\#D=\#L=m$.
The following are equivalent:
\begin{enumerate}[(i)]
\item $(q,D,L)$ is a Hadamard triple;
\item the elements of $L$ are distinct modulo $q$, and replacing them
by their residues gives a Hadamard triple;
\item $L$ is a spectrum of $\delta_{q^{-1}D}$;
\item $\sum_{\ell\in L}|m_D((\xi+\ell)/q)|^2=1$ for all $\xi\in\mathbb R$.
\end{enumerate}
Moreover, if $t\in\mathcal T$, then
$t(L+qL+\cdots+q^{n-1}L)$ is a spectrum of $\mu_n$.
\end{lem}

For the infinite measure, we will need Fourier lower bounds that
remain valid under small perturbations. The following continuity
property supplies this stability.
\begin{lem}[{\cite[Lemma 2.2]{AFL19}}]\label{thm2.2}
Let $\mu$ be a compactly supported Borel probability measure on
$\mathbb R$. Then $\widehat\mu$ is uniformly continuous on $\mathbb R$.
\end{lem}

Finite-level orthogonality alone does not guarantee completeness
of the limiting system. The next criterion supplements nested finite
spectra with a uniform lower bound for the tail Fourier transform.
It follows from the argument of \cite[Theorem 2.3]{AHH19}, together
with Lemma \ref{thm2.2}, and will be applied along a subsequence
on which a prescribed multiplier is fixed.
\begin{thm}\label{thm2.3}
	Let $\mu:=\mu_{q,D}$ be given as \eqref{(1.1)}, and let $b_j\to\infty$ be strictly increasing positive integers. Suppose
$0\in\Lambda_{b_j}\subset\Lambda_{b_{j+1}}$, that $\Lambda_{b_j}$ is a
spectrum of $\mu_{b_j}$, and that
\[
 \inf_j\inf_{\lambda\in\Lambda_{b_j}}
              |\widehat\mu(q^{-b_j}\lambda)|^2>0.
\]
Then $\bigcup_j\Lambda_{b_j}$ is a spectrum of $\mu$.
\end{thm}

To meet the tail condition in Theorem \ref{thm2.3}, we use the
following selection result from \cite{Lu26}. It supplies two integer
translations at every point of a compact interval, with a positive
lower bound independent of that point.
\begin{lem}[\cite{Lu26}]\label{thm2.4}
Let $\mu:=\mu_{q,D}$ and $\mathcal{T}$ be given as \eqref{(1.1)} and \eqref{(1.2)}, respectively, where $(q,D,L)$ forms a Hadamard triple. 
Assume $2\leq m<q$. For every $t\in\mathcal T\cap\mathbb N$ there are
$\epsilon_t>0$ and an integer $N_t\geq1$ such that, for each
$x\in[0,t]$, there are two integers
$0\leq u_1(x,t)<u_2(x,t)<q^{N_t}$ satisfying
\[
              |\widehat\mu(x+t u_j(x,t))|\geq2\epsilon_t,
                 \qquad j=1,2.
\]
\end{lem}

The strict inequality $m<q$ is part of the hypotheses of this
selection result. The constants may depend on the multiplier.
To use the result for every positive multiplier, we arrange that
each one is treated infinitely many times in the construction.

\begin{de}\label{de2.1}
A sequence $(a_k)$ in a countable set $\mathcal A$ is called a
\emph{recurrent enumeration} if each element of $\mathcal A$
occurs infinitely often.
\end{de}
We use a recurrent enumeration of $\mathcal T_+:=\mathcal T\cap\mathbb N$.
Negative multipliers will be recovered by complex conjugation.

We now combine Lemmas \ref{thm2.2} and \ref{thm2.4} into a form
adapted to the recursive construction. Besides stability under
perturbations, we arrange that the first translation is zero near
the origin, so previously chosen frequencies can be retained.
\begin{lem}\label{thm2.5}
			Let $\mu:=\mu_{q,D}$ and $\mathcal{T}$ be given as \eqref{(1.1)} and \eqref{(1.2)}, respectively, where $(q,D,L)$ forms a Hadamard triple. 
For each $t\in\mathcal T_+:=\mathcal{T}\cap\mathbb N$, there are numbers
$0<\epsilon_t<1/2$, $0<\delta_t<1/4$, and an integer $N_t\geq1$
with $tq^{-N_t}<\delta_t/4$ such that there exists an integer-valued map
\[
\eta_t:[0,t]\times\{1,2\}\longrightarrow
\{0,1,\ldots,q^{N_t}-1\}
\]
satisfying
\[
\eta_t(x,1)<\eta_t(x,2)\qquad(0\leq x\leq t)
\]
and 
\[
 |\widehat\mu(x+t\eta_t(x,j)+y)|\geq\epsilon_t
 \quad (|y|<\delta_t,\ j=1,2).
\]
The map and constants can be chosen to satisfy
$\eta_t(x,1)=0$ for $0\leq x\leq\delta_t$ and
\[
|\widehat\mu(v)|\geq\frac12
\qquad\text{whenever }|v|<2\delta_t.
\]
In particular, $b_t:=\eta_t(0,2)$ is a positive integer.
\end{lem}
\begin{proof}
Fix $t$. Lemma \ref{thm2.4} supplies a number $\epsilon_t^0>0$,
an integer $N_t^0$, and two integers $u_1(x)<u_2(x)$ such that
$|\widehat\mu(x+tu_j(x))|\geq2\epsilon_t^0$.
Choose $0<\epsilon_t<\min\{\epsilon_t^0,1/4\}$.
By uniform continuity choose $0<\delta_t<1/4$ sufficiently small that
\[
 |\widehat\mu(z+y)-\widehat\mu(z)|<\epsilon_t
       \quad(z\in\mathbb R,\ |y|<\delta_t).
\]
Decrease $\delta_t$ again, if needed, to have
$|\widehat\mu(v)|\geq1/2$ for $|v|<2\delta_t$.
For either of the original choices, the triangle inequality gives
\[
\begin{split}
 |\widehat\mu(x+tu_j(x)+y)|
 &\geq |\widehat\mu(x+tu_j(x))|
       -|\widehat\mu(x+tu_j(x)+y)-\widehat\mu(x+tu_j(x))|\\
 &\geq 2\epsilon_t^0-\epsilon_t\geq\epsilon_t.
\end{split}
\]
For $x>\delta_t$ keep these choices. For $0\leq x\leq\delta_t$
put $\eta_t(x,1)=0$ and $\eta_t(x,2)=u_2(x)$.
The second is positive because $u_2(x)>u_1(x)\geq0$.
The first satisfies the required bound since
$|x+y|<2\delta_t$ and $|\widehat\mu(x+y)|\geq1/2>\epsilon_t$.
Finally enlarge $N_t^0$ to an integer $N_t$ with
$tq^{-N_t}<\delta_t/4$. All selected integers still lie below
$q^{N_t}$, and $b_t=\eta_t(0,2)>0$ as required.
\end{proof}
For each $t\in\mathcal T_+$, we henceforth fix
$\epsilon_t,\delta_t,N_t$, and $\eta_t$ satisfying all
the properties in Lemma \ref{thm2.5}.
\section{Word spaces and spectral extension}\label{sect.3}
This section explains how to extend a prescribed family of digit
frequencies to a spectrum whose dilates by every $t\in\mathcal T$ are also
spectra. We first normalize the digit set and establish the finite-stage
orthogonality and Fourier estimates. We then give a recursive construction
that retains the prescribed frequencies and adds a $q$-lacunary set of
frequencies to ensure completeness. The binary encodings distinguish the
resulting spectra; Section \ref{sect4} will show that the added frequencies
do not affect the prescribed upper $s$-Beurling density for $s>0$.

\begin{lem}\label{thm3.1}
A Hadamard triple with $m<q$ admits a frequency digit set
$L'\subset\{0,\ldots,q-2\}$ containing zero. It has exactly the same
multiplier set \eqref{(1.2)} as the original digit set $L$.
\end{lem}
\begin{proof}
By Lemma \ref{thm2.1}, replace every element of $L$ by its residue
in $\{0,\ldots,q-1\}$. The resulting residue set is nonempty and
proper because $1\leq m<q$.

There is an element $c\in L$ whose predecessor $c-1$ modulo $q$
does not belong to $L$. Indeed, if every element had its predecessor
in $L$, starting from one element and repeatedly taking predecessors
would place all $q$ residues in $L$, a contradiction. Define
\[
                 L'=(L-c)\pmod q
\]
using residues in $\{0,\ldots,q-1\}$. The element $c$ produces zero,
whereas $q-1$ cannot occur because that would require $c-1\in L$.
Thus $0\in L'\subset\{0,\ldots,q-2\}$.

Translation by $-tc$ multiplies row $d$ of the matrix
for $(q,D,tL)$ by $e^{-2\pi i dtc/q}$, a scalar of modulus one.
Replacing a translated digit by a congruent digit adds an integer
multiple of $2\pi i$ to every exponent, leaving the entries unchanged.
Both operations preserve unitarity in both directions. For every
integer $t$, $(q,D,tL)$ is therefore Hadamard if and only if
$(q,D,tL')$ is Hadamard. This proves equality of the multiplier sets.
\end{proof}
Henceforth use this normalization and write
\[
 L=\{0=l_0<l_1<\cdots<l_{m-1}\}\subset\{0,\ldots,q-2\},
 \qquad\Omega=\{0,\ldots,m-1\}.
\]
For $J=j_1\cdots j_n\in\Omega^n$, put
\[
 T_n(J)=\sum_{r=1}^n l_{j_r}q^{r-1}.
\]
Words are read from the least significant digit. For example, for
$q=4$ and $L=\{0,1\}$, the word $101$ has value
$T_3(101)=1+0\cdot4+1\cdot4^2=17$, whereas $T_3(110)=5$.
Thus a prefix specifies the low digits, and a suffix specifies the
high digits; deleting a suffix removes the most significant digits.
Appending zero digits does not change the numerical value of a word:
$T_{n+r}(J0^r)=T_n(J)$ for $r\geq0$. We set $T_0(\emptyset)=0$.

The next lemma provides the Fourier lower bounds used to prove
completeness. We use $\eta_t$, $b_t$, $\epsilon_t$, and $\delta_t$
from Lemma \ref{thm2.5}.
\begin{lem}\label{thm3.2}
For $t\in\mathcal T_+$, $J\in\Omega^n$, and $i\in\{1,2\}$, put
\begin{equation}\tag{3.1}\label{(3.1)}
 \psi_n^{(t)}(J,i)=T_n(J)+q^n\eta_t\bigl(tq^{-n}T_n(J),i\bigr).
\end{equation}
Then $\psi_n^{(t)}(J,i)\equiv T_n(J)\pmod{q^n}$ and
\[
 |\widehat\mu(y+tq^{-n}\psi_n^{(t)}(J,i))|\geq\epsilon_t
                         \qquad(|y|<\delta_t).
\]
If $tq^{-n}T_n(J)\leq\delta_t$, then
$\psi_n^{(t)}(J,1)=T_n(J)$. Moreover, if $x\in[0,t]$,
$|y|<\delta_t$, $u=\eta_t(x,j)$ for some $j\in\{1,2\}$, and
$z=x+y+tu$, then, for every sufficiently large integer $r$,
\[
                  |\widehat\mu(z+tq^r b_t)|\geq\epsilon_t^2.
\]
The threshold for $r$ may be increased arbitrarily.
\end{lem}
\begin{proof}
The digits lie between zero and $q-2$, so
\[
 0\leq T_n(J)\leq(q-2)\frac{q^n-1}{q-1}<q^n.
\]
Thus $x=tq^{-n}T_n(J)$ belongs to $[0,t]$, the domain of $\eta_t$.
Formula \eqref{(3.1)} immediately gives the congruence modulo $q^n$.
Moreover,
\[
 y+tq^{-n}\psi_n^{(t)}(J,i)=y+x+t\eta_t(x,i).
\]
Its Fourier modulus is at least $\epsilon_t$ by Lemma \ref{thm2.5}.
If $x\leq\delta_t$, that lemma allows $\eta_t(x,1)=0$, proving the
zero-correction assertion. 

For the last part, Lemma \ref{thm2.5} gives
$|\widehat\mu(z)|\geq\epsilon_t$ for either choice $j\in\{1,2\}$.
Choose $r$ so large that $|q^{-r}z|<\delta_t$.
Since $tb_t$ is an integer and $\widehat{\mu_r}$ is $q^r$-periodic,
\[
\begin{split}
 |\widehat\mu(z+tq^rb_t)|
 &=|\widehat{\mu_r}(z+tq^rb_t)|
                       |\widehat\mu(q^{-r}z+tb_t)|=|\widehat{\mu_r}(z)|\,|\widehat\mu(q^{-r}z+tb_t)|.
\end{split}
\]
From \eqref{(2.1)} it follows that
$|\widehat{\mu_r}(z)|\geq|\widehat\mu(z)|\geq\epsilon_t$.
The other factor is at least $\epsilon_t$, using $x=0$, the second
choice $b_t$, and perturbation $q^{-r}z$ in Lemma \ref{thm2.5}.
Multiplication gives $\epsilon_t^2$. All larger $r$ also satisfy
$|q^{-r}z|<\delta_t$, so any additional finite lower bound on $r$
can be imposed.
\end{proof}

\begin{de}\label{de3.1}
A family $\mathcal J_n\subset\Omega^n$ is an \emph{increasing word
sequence} if $0^n\in\mathcal J_n$ and
\[
 \mathcal J_n0\subset\mathcal J_{n+1},\qquad
 (\Omega^n\setminus\mathcal J_n)0
                         \subset\Omega^{n+1}\setminus\mathcal J_{n+1}.
\]
It is \emph{prefix closed} if in addition
$\mathcal J_{n+1}\subset\mathcal J_n\Omega$.
Set $\mathcal J_0=\{\emptyset\}$ and
$\Lambda(\mathcal J)=\bigcup_nT_n(\mathcal J_n)$.
\end{de}
Prefix closure is an additional hypothesis wherever used below. All
word sequences constructed in Section \ref{sect4} have this property.

\begin{lem}\label{thm3.3}
Let $0=m_0<m_1<m_2<\cdots$, $n_k=m_k-m_{k-1}$, and
$\lambda_0(\emptyset)=0$. Define integers recursively by
\[
 \lambda_k(PJ)=\lambda_{k-1}(P)+q^{m_{k-1}}T_{n_k}(J)
                              +q^{m_k}c_k(P,J),
 \quad P\in\Omega^{m_{k-1}},\ J\in\Omega^{n_k},
\]
where $c_k(P,J)\in\mathbb Z$ and $c_k(P,0^{n_k})=0$.
The other carries may depend on the
\emph{whole node} $(P,J)$ and on all earlier choices. Then
$\Lambda_k=\{\lambda_k(W):W\in\Omega^{m_k}\}$ are nested, and
$t\Lambda_k$ is a spectrum of $\mu_{m_k}$ for every $t\in\mathcal T$.
\end{lem}
\begin{proof}
We prove the lemma by induction on $k$. A zero suffix gives
\[
 \lambda_k(P0^{n_k})
 =\lambda_{k-1}(P)+q^{m_{k-1}}\cdot0+q^{m_k}\cdot0
 =\lambda_{k-1}(P).
\]
Thus every old frequency is retained and zero is retained at every
stage. This proves nestedness.

Fix $t\in\mathcal T$. Suppose the conclusion has been proved through
stage $k-1$, and compare different nodes $PJ$ and $P'J'$ at stage $k$.
Keep $n_k=m_k-m_{k-1}$.

\smallskip\noindent\textbf{Case 1: $P\ne P'$.}
The recursive definition gives an integer $z$ with
\[
 t(\lambda_k(PJ)-\lambda_k(P'J'))
     =t(\lambda_{k-1}(P)-\lambda_{k-1}(P'))+q^{m_{k-1}} z.
\]
By $q^{m_{k-1}}$-periodicity and the induction hypothesis,
\[
 \widehat{\mu_{m_{k-1}}}\bigl(t(\lambda_k(PJ)-\lambda_k(P'J'))\bigr)
 =\widehat{\mu_{m_{k-1}}}\bigl(t(\lambda_{k-1}(P)-\lambda_{k-1}(P'))\bigr)=0.
\]
This is a factor of the Fourier transform for $\mu_{m_k}$, so the
latter vanishes as well.

\smallskip\noindent\textbf{Case 2: $P=P'$ and $J\ne J'$.}
Let $z=c_k(P,J)-c_k(P,J')$. Now
\[
 t(\lambda_k(PJ)-\lambda_k(PJ'))
      =tq^{m_{k-1}}\bigl(T_{n_k}(J)-T_{n_k}(J')+q^{n_k} z\bigr).
\]
The product of Fourier factors from levels $m_{k-1}+1$ through $m_{k-1}+n_k$ is
\[
\begin{split}
 \prod_{r=1}^{n_k} m_D\bigl(q^{-r}t(T_{n_k}(J)-T_{n_k}(J')+q^{n_k}z)\bigr)&=\widehat{\mu_{n_k}}\bigl(t(T_{n_k}(J)-T_{n_k}(J')+q^{n_k}z)\bigr)\\
 &=\widehat{\mu_{n_k}}\bigl(t(T_{n_k}(J)-T_{n_k}(J'))\bigr)=0.
\end{split}
\]
Periodicity gives the penultimate equality and Lemma \ref{thm2.1}
gives the last.

For $k=1$ there is only the empty parent, so Case 2 starts the
induction. In either case the exponential inner product is zero.
Different nodes cannot give the same frequency, since that would
give inner product one. There are exactly $m^{m_k}$ nodes and,
by Lemma \ref{thm2.1}, exactly $m^{m_k}$ atoms of $\mu_{m_k}$.
The orthonormal system therefore has the full dimension and is a
basis. 
\end{proof}

For the next construction, $I=i_1i_2\cdots\in\{1,2\}^{\mathbb N}$
is a \emph{binary encoding} of the choices of frequencies; the two
symbols $1$ and $2$ encode the two alternatives at each stage.
Write $I|_k=i_1\cdots i_k$. More generally, for a finite binary
encoding $\iota=i_1\cdots i_k\in\{1,2\}^k$, write
$\iota^-=i_1\cdots i_{k-1}$. At level zero both the digit word and
the binary encoding are the empty word $\emptyset$.
We call a word $J\in\mathcal J_n$ a \emph{core word}.
At stage $k$, the values $\lambda_k(J,\iota)$ with
$J\in\mathcal J_{m_k}$ are called \emph{core frequency values}.
\begin{thm}[Recursive spectral extension]\label{thm3.4}
Let $(a_k)$ be a recurrent enumeration of $\mathcal T_+$ and let
$(\mathcal J_n)$ be a prefix-closed increasing word sequence.
Put $m_0=0$ and $\lambda_0(\emptyset,\emptyset)=0$.
Construct the levels recursively, with
\[
  m_{k-1}<\gamma_k<m_k<\gamma_{k+1},\qquad n_k=m_k-m_{k-1}.
\]
At the start of stage $k$, the finite sets $\Lambda_{k-1}(\iota)$
from stage $k-1$ have been defined, where
\[
 \Lambda_{k-1}(\iota)
 =\{\lambda_{k-1}(P,\iota):P\in\Omega^{m_{k-1}}\},
                     \qquad\iota\in\{1,2\}^{k-1}.
\]
Define the maximum over all these sets by
\[
 M_{k-1}
 :=\max\!\bigcup_{\iota\in\{1,2\}^{k-1}}\Lambda_{k-1}(\iota)
 =\max_{\substack{\iota\in\{1,2\}^{k-1}\\P\in\Omega^{m_{k-1}}}}
                          \lambda_{k-1}(P,\iota),
 \qquad M_0=0.
\]
Suppose the choices at every stage satisfy
\[
 \mathcal J_{m_k}=\mathcal J_{\gamma_k}0^{m_k-\gamma_k},\qquad
 a_kq^{\gamma_k-m_k}<\delta_{a_k}/4,\qquad
 a_kq^{-m_k}M_{k-1}<\delta_{a_k}/4.
\]
Then the integer corrections can be chosen so that, for each binary
encoding $I\in\{1,2\}^{\mathbb N}$, the set
\[
                 \Lambda(I)=\bigcup_{k\geq1}\Lambda_k(I|_k)
\]
has the following properties:
\begin{enumerate}[(i)]
\item $t\Lambda(I)$ is a spectrum of $\mu$ for every $t\in\mathcal T$;
\item $\Lambda(I)=\Lambda(\mathcal J)\cup E(I)$, where explicitly
\[
 E(I)=\bigcup_{k\geq1}
 \left\{\lambda_k(PJ,I|_k):
 \begin{array}{l}
 P\in\Omega^{m_{k-1}},\quad J\in\Omega^{n_k}\setminus\{0^{n_k}\},\\
 PJ\notin\mathcal J_{m_k}
 \end{array}\right\}.
\]
The set $E(I)$ is disjoint from $\Lambda(\mathcal J)$, consists of
positive integers, and is $q$-lacunary;
\item $\Lambda(I)\ne\Lambda(I')$ whenever $I\ne I'$.
\end{enumerate}
Here a positive set is $q$-lacunary if its increasing enumeration
$(e_j)$ satisfies $e_{j+1}\geq q e_j$.
\end{thm}
\begin{proof}
We first define the corrections and the finite sets, including all
binary encodings, and then verify the three conclusions.

\smallskip\noindent\textbf{Step 1: the nodes that create new additional frequencies.}
Suppose stage $k-1$ is complete for all
$\iota^-\in\{1,2\}^{k-1}$. The sets in the definition of $M_{k-1}$
are nonempty and finite: there are $2^{k-1}$ binary encodings and
$m^{m_{k-1}}$ digit words for each. Thus $M_{k-1}$ is known and
finite before $\gamma_k,m_k$ are chosen. Nestedness, verified below,
will also ensure that it bounds every frequency from earlier stages.

For $\iota\in\{1,2\}^k$ and a digit word
$PJ\in\Omega^{m_k}$, where $P\in\Omega^{m_{k-1}}$ and
$J\in\Omega^{n_k}$, define
\[
 \lambda_k(PJ,\iota)
 =\lambda_{k-1}(P,\iota^-)+q^{m_{k-1}}T_{n_k}(J)
                                      +q^{m_k}c_k(P,J,\iota).
\]
For every parent $P$ and encoding $\iota$, set
$c_k(P,0^{n_k},\iota)=0$. Thus a zero suffix
will retain its parent's value. Also set $c_k(P,J,\iota)=0$ whenever
$PJ\in\mathcal J_{m_k}$. The only remaining triples are
\[
 \mathscr N_k=
 \left\{(\iota,P,J):
 \begin{array}{l}
 \iota\in\{1,2\}^k,\quad P\in\Omega^{m_{k-1}},\\
 J\in\Omega^{n_k}\setminus\{0^{n_k}\},\quad PJ\notin\mathcal J_{m_k}
 \end{array}\right\}.
\]
These, and only these, receive a large positive correction at stage
$k$. 

For a triple in $\mathscr N_k$, put
\[
 x_k(J)=a_kq^{-n_k}T_{n_k}(J),\qquad
 y_k(P,\iota^-)=a_kq^{-m_k}\lambda_{k-1}(P,\iota^-).
\]
Because $0\leq T_{n_k}(J)<q^{n_k}$ and by the assumed bound on
$M_{k-1}$, these satisfy
\[
 0\leq x_k(J)<a_k,\qquad
 0\leq y_k(P,\iota^-)<\delta_{a_k}/4.
\]
For each such triple, choose an integer
$r_k(\iota,P,J)\geq1$ and put
\[
 \varphi_k(P,J,\iota)
       =q^{r_k(\iota,P,J)}b_{a_k},
             \qquad b_{a_k}=\eta_{a_k}(0,2)>0,
\]
For $\iota=\iota_1\cdots\iota_k$ and $(\iota,P,J)\in\mathscr N_k$, define
\[
 \sigma_k(P,J,\iota)=
 \begin{cases}
 \iota_k,&P=0^{m_{k-1}},\quad J=0^{n_k-1}1,\\
 1,&\text{otherwise}.
 \end{cases}
\]
The distinguished word $0^{m_k-1}1$ lies outside $\mathcal J_{m_k}$,
since $m_k>\gamma_k$ and every core word has zero final digit.
Its suffix is nonzero, so the corresponding triple belongs to
$\mathscr N_k$ for each encoding.
Set
\begin{equation}\tag{3.2}\label{(3.2)}
 c_k(P,J,\iota)
 =\eta_{a_k}\bigl(x_k(J),\sigma_k(P,J,\iota)\bigr)
   +\varphi_k(P,J,\iota),
 \qquad (\iota,P,J)\in\mathscr N_k.
\end{equation}
Thus the distinguished node uses the choice indexed by $\iota_k$,
while every other triple uses the first choice.
Every newly assigned frequency value is a positive integer.

The exponent will depend on the whole triple, including the binary
encoding and the parent, so all these values can be separated.

\smallskip\noindent\textbf{Step 2: an explicit successive choice of the exponents.}
Give the finite set $\mathscr N_k$ any fixed order and let
$\rho_k:\mathscr N_k\to\{1,\ldots,\#\mathscr N_k\}$ be its rank map.
Suppose the first $j-1$ triples at this stage have been assigned.
Let $B_{k,j-1}$ be the maximum of $M_{k-1}$ and the frequency
values already assigned to the first $j-1$ triples at stage $k$:
\[
 B_{k,j-1}:=\max\left(
 \{M_{k-1}\}\cup
 \left\{\lambda_k(PJ,\iota):
 \begin{array}{l}
 (\iota,P,J)\in\mathscr N_k,\quad
 \rho_k(\iota,P,J)\leq j-1
 \end{array}\right\}\right).
\]
All these values have already been determined. For $j=1$ the second
set is empty, so $B_{k,0}=M_{k-1}$; in particular $B_{1,0}=0$.
By nestedness, $M_{k-1}$ bounds all frequencies assigned at earlier
stages, for all binary encodings.

For the triple $(\iota,P,J)$ of rank $j$, write
\[
 z_k(P,J,\iota)
 =x_k(J)+y_k(P,\iota^-)
  +a_k\eta_{a_k}\bigl(x_k(J),\sigma_k(P,J,\iota)\bigr).
\]
Choose $r_k(\iota,P,J)$ to be the least positive integer $r$ satisfying
both
\[
 q^{-r}|z_k(P,J,\iota)|<\delta_{a_k}
\]
and
\[
\begin{split}
 &\lambda_{k-1}(P,\iota^-)+q^{m_{k-1}}T_{n_k}(J)\\
 &\quad+q^{m_k}\left(
 \eta_{a_k}\bigl(x_k(J),\sigma_k(P,J,\iota)\bigr)+q^rb_{a_k}
 \right)
 >\max\{qB_{k,j-1},M_{k-1},1\}.
\end{split}
\]
Such an integer exists: with all preceding data fixed, the first
left-hand side tends to zero and the frequency in the second
inequality tends to infinity as $r\to\infty$. This rule specifies the exponents once the selection maps and the
finite ordering have been fixed.
It also guarantees that every new additional value is greater than
all earlier-stage frequencies and greater than $q$ times every
previously assigned additional value.
In particular, for any two triples in $\mathscr N_k$,
\[
 \rho_k(\iota,P,J)>\rho_k(\iota',P',J')
 \quad\Longrightarrow\quad
 \lambda_k(PJ,\iota)>q\lambda_k(P'J',\iota').
\]
This holds for all encodings, including distinguished nodes, because
the size condition uses the actual choice $\sigma_k(P,J,\iota)$.
Together with the bound on earlier stages, it gives global strict
separation of all newly assigned additional values.

Once all triples in $\mathscr N_k$ have been assigned, set
\[
 \Lambda_k(\iota)=\{\lambda_k(W,\iota):W\in\Omega^{m_k}\},\qquad
 M_k=\max_{\iota\in\{1,2\}^{k},\;W\in\Omega^{m_k}}
                                      \lambda_k(W,\iota).
\]
These are finite data for the next stage. In particular the
construction of stage $k$ does not use the unknown future value $M_k$
in choosing its own word length.

\smallskip\noindent\textbf{Step 3: finite spectrality and the precise description of $E(I)$.}
All corrections are integers and zero suffixes have correction zero.
Lemma \ref{thm3.3} therefore gives, for each binary encoding $\iota$,
\[
 0\in\Lambda_{k-1}(\iota^-)\subset\Lambda_k(\iota),\qquad
 t\Lambda_k(\iota)\text{ is a spectrum of }\mu_{m_k}
                              \quad(t\in\mathcal T).
\]
For a core node, prefix closure implies that its parent is a core
node. Induction, starting from zero, gives
\[
 \lambda_k(W,\iota)=T_{m_k}(W)
                 \quad(W\in\mathcal J_{m_k},\ \iota\in\{1,2\}^k).
\]
These values do not depend on the binary encoding.

Define the new additional set at stage $k$ for a fixed finite
encoding by
\[
 E_k(\iota)=\{\lambda_k(PJ,\iota):(\iota,P,J)\in\mathscr N_k\}.
\]
By the three cases in Step 1, the finite sets satisfy the exact
recursive identity
\[
 \Lambda_k(\iota)
 =\Lambda_{k-1}(\iota^-)
                   \cup T_{m_k}(\mathcal J_{m_k})\cup E_k(\iota).
\]
Indeed a zero suffix produces the first set, a core node produces
the second, and a nonzero suffix outside the core produces the third.
Taking unions along an infinite binary encoding gives
\[
 \Lambda(I)=\Lambda(\mathcal J)\cup E(I),\qquad
                      E(I):=\bigcup_{k\geq1}E_k(I|_k).
\]
Here every core word at an intermediate length extends by zeros to
some $m_k$, which justifies the union of the core terms.

We also verify that no additional value can become a core value
later. If $\lambda\in E_k(\iota)$, its node is distinct from every
core node at stage $k$. They belong to the same finite orthogonal
set for the multiplier $1\in\mathcal T$. Their frequencies cannot
be congruent modulo $q^{m_k}$, since
\[
                   \widehat{\mu_{m_k}}(q^{m_k}z)=1
                                      \qquad(z\in\mathbb Z).
\]
Every limiting core value reduces modulo $q^{m_k}$ to an element
of $T_{m_k}(\mathcal J_{m_k})$, by prefix closure and zero extension.
Thus $\lambda\notin\Lambda(\mathcal J)$, and
$E(I)\cap\Lambda(\mathcal J)=\varnothing$.

\smallskip\noindent\textbf{Step 4: the tail Fourier estimates and completeness.}
At stage $k$, we estimate the tail Fourier transform for the
multiplier $a_k$. For a retained frequency, the hypotheses give
\[
 0\leq a_kq^{-m_k}\lambda_{k-1}(P,\iota^-)
                  \leq a_kq^{-m_k}M_{k-1}<\delta_{a_k}/4.
\]
For a core frequency, the zero digits after $\gamma_k$ give
\[
 0\leq a_kq^{-m_k}T_{m_k}(W)
                  <a_kq^{\gamma_k-m_k}<\delta_{a_k}/4.
\]
Thus, for every $\lambda\in\Lambda_{k-1}(\iota^-)
\cup T_{m_k}(\mathcal J_{m_k})$, we have
$|\widehat\mu(a_kq^{-m_k}\lambda)|\geq1/2$ by the choice of
$\delta_{a_k}$ in Lemma \ref{thm2.5}.
For a new additional frequency in $E_k(\iota)$, \eqref{(3.2)} gives
\[
 a_kq^{-m_k}\lambda_k(PJ,\iota)
       =z_k(P,J,\iota)+a_kq^{r_k(\iota,P,J)}b_{a_k}.
\]
Since $\sigma_k(P,J,\iota)\in\{1,2\}$ and
$|y_k(P,\iota^-)|<\delta_{a_k}/4$, Lemma \ref{thm2.5} yields
$|\widehat\mu(z_k(P,J,\iota))|\geq\epsilon_{a_k}$.
Step 2 ensures $q^{-r_k}|z_k|<\delta_{a_k}$ for this same $z_k$.
The factorization in the proof of Lemma \ref{thm3.2}, valid for
either selection value, therefore gives
\[
\begin{split}
 &|\widehat\mu(a_kq^{-m_k}\lambda_k(PJ,\iota))|\geq|\widehat\mu(z_k(P,J,\iota))|
 \left|\widehat\mu\bigl(q^{-r_k(\iota,P,J)}z_k(P,J,\iota)
                                      +a_kb_{a_k}\bigr)\right|\geq\epsilon_{a_k}^2.
\end{split}
\]
Since $0<\epsilon_{a_k}<1/2$, the three cases imply
\[
 \inf_{\lambda\in\Lambda_k(\iota)}
         |\widehat\mu(a_kq^{-m_k}\lambda)|^2\geq\epsilon_{a_k}^4
                         \qquad(\iota\in\{1,2\}^k).
\]

Fix $t\in\mathcal T_+$. By recurrence, the indices satisfying
$a_k=t$ form an infinite sequence
$k_1<k_2<\cdots$. Apply Theorem \ref{thm2.3} with
$b_j=m_{k_j}$, finite spectra $t\Lambda_{k_j}(I|_{k_j})$, and the
uniform constant $\epsilon_t^4>0$. Nestedness and $k_j\to\infty$ imply
that their union is $t\Lambda(I)$, so it is a spectrum of $\mu$.
Complex conjugation takes its exponential basis to the basis with
frequencies $-t\Lambda(I)$. Since $\mathcal T=-\mathcal T$, this
proves (i) for every multiplier.

\smallskip\noindent\textbf{Step 5: sparsity and distinctness.}
List all additional values in their order of assignment, first by
stage and then by the rank $\rho_k$. Step 2 makes each one strictly
greater than $q$ times the preceding one. Thus the global set of
additional values is $q$-lacunary, and so is its subset $E(I)$.
Together with Step 3 this proves (ii).

For (iii), suppose $I\ne I'$ and let $k$ be their first differing
position. At the node $0^{m_k-1}1$, choose the additional frequency 
$\lambda_*$ belonging to $I|_k$. It is retained in $\Lambda(I)$ at
every later stage. The value $\lambda_*$ is assigned to a triple whose binary
encoding is $I|_k$. Since $I|_k\ne I'|_k$, this triple does
not occur along $I'$. Step 2 assigns distinct frequency values
to all triples, both within each stage and across different
stages. Hence $\lambda_*\notin E(I')$. Step 3 gives
$\lambda_*\notin\Lambda(\mathcal J)$ as well.
Consequently $\lambda_*\notin\Lambda(I')$ and the two spectra differ.
\end{proof}
The size conditions in Theorem \ref{thm3.4} are imposed recursively.
In Section \ref{sect4}, we choose each new core stage after $M_{k-1}$
is known and verify these conditions while preserving the prescribed density.

\begin{lem}\label{thm3.5}
If $E\subset(0,\infty)$ is $b$-lacunary for some $b>1$, then
$M_E(h)=O(1+\log h)$ as $h\to\infty$. Consequently
$D_u^+(E)=0$ for every $u>0$.
\end{lem}
\begin{proof}
If $E$ is empty or finite the lemma is immediate. Otherwise
write its increasing enumeration as $0<e_1<e_2<\cdots$.
Suppose $B(x,h)$ contains $r\geq1$ elements and let $e_j$ be the
first of them. Their last element is at least $b^{r-1}e_j$.
Both endpoints lie in an interval of length $2h$, so
\[
 (b^{r-1}-1)e_1\leq (b^{r-1}-1)e_j<2h.
\]
Solving this inequality for $r$ gives
\[
 \#(E\cap B(x,h))\leq1+\log_b(1+2h/e_1).
\]
The right-hand side is independent of $x$, hence also bounds
$M_E(h)$. For $h\geq1$ it is at most $C(1+\log h)$ for a constant
depending on $b,e_1$. For each $u>0$,
\[
 0\leq\frac{M_E(h)}{h^u}
       \leq C\frac{1+\log h}{h^u}\longrightarrow0.
\]
Taking limsup proves $D_u^+(E)=0$.
Adjoining zero changes each interval count by at most one.
\end{proof}

\section{Prescribed Beurling density}\label{sect4}
In this section, only the digit assumptions
$0\in L\subset\{0,\ldots,q-2\}$ and $m=\#L\geq2$ are needed.
Write
\[
 \Gamma_0=\{0\},\qquad\Gamma_n=T_n(\Omega^n),\qquad
 \Gamma=\bigcup_{n\geq1}\Gamma_n,\qquad r_0=\log m/\log q.
\]

\begin{lem}\label{thm4.1}
Let $(\mathcal J_n)$ be prefix closed and $\Lambda(\mathcal{J})$ be given in Definition \ref{de3.1}.
If $h\geq1$ and $q^{n_h}\leq2h<q^{n_h+1}$ for some integer $n_h\in\mathbb N$, then
\begin{equation}\tag{4.1}\label{(4.1)}
 M_{\Lambda(\mathcal J)}(h)
                   =M_{T_{n_h+2}(\mathcal J_{n_h+2})}(h).
\end{equation}
The supremum on the right is attained by a center in
$[0,q^{n_h+2})$.
\end{lem}
\begin{proof}
Let $\lambda,\lambda'\in\Lambda(\mathcal J)$ be distinct. Write
$\lambda=\sum_{r\geq1}w_rq^{r-1}$ and
$\lambda'=\sum_{r\geq1}w'_rq^{r-1}$, where $w_r,w'_r\in L$
and both digit sequences are eventually zero.
Pad with zeros so they have equal length, and let $j$ be their
highest differing position. The coefficient at that position has
absolute value at least one, while all lower differences have
absolute value at most $q-2$. Consequently
\[
\begin{split}
 |\lambda-\lambda'|
 &\geq q^{j-1}-(q-2)\sum_{r=0}^{j-2}q^r=q^{j-1}-(q-2)\frac{q^{j-1}-1}{q-1}
 =\frac{q^{j-1}+q-2}{q-1}.
\end{split}
\]
If $j\geq n_h+3$, this is at least
$q^{n_h+2}/(q-1)>q^{n_h+1}>2h$.
Two points of $B(x,h)$ cannot have such a distance.
Therefore all points in $B(x,h)\cap\Lambda(\mathcal J)$ have identical
digits above position $n_h+2$.

For a nonempty intersection write this common tail as
$v=\sum_{r\geq n_h+3}w_rq^{r-1}$. For every selected $\lambda\in\Lambda(\mathcal{J})\cap B(x,h)$,
\[
 \lambda-v\in T_{n_h+2}(\mathcal J_{n_h+2}),\qquad
                 \lambda-v\in B(x-v,h).
\]
So
\[
 \#(B(x,h)\cap\Lambda(\mathcal J))
 \leq\#(B(x-v,h)\cap T_{n_h+2}(\mathcal J_{n_h+2}))
 \leq M_{T_{n_h+2}(\mathcal J_{n_h+2})}(h).
\]
Taking the
supremum over $x$ gives one direction of \eqref{(4.1)}.
The other direction follows from
$T_{n_h+2}(\mathcal J_{n_h+2})\subset\Lambda(\mathcal J)$.

For a nonempty finite set the possible interval counts belong to a
finite set of integers, so their supremum is attained.
If a maximizing center lies outside the convex hull of the selected
points, move it to the nearest endpoint of that hull. This only
reduces its distance from the selected points. The new center remains
maximizing and lies in $[0,q^{n_h+2})$, since all the finite digit
values lie in that interval.
\end{proof}

\begin{re}\label{re4.2}
For $\mathcal J_n=\Omega^n$, Lemma \ref{thm4.1} gives
$M_\Gamma(h)=M_{\Gamma_{n_h+2}}(h)$.
\end{re}

\begin{lem}\label{thm4.2}
There is a constant $C=C(q,m)$ such that
$M_\Gamma(h)\leq C h^{r_0}$ for all $h\geq1$. In particular,
\[
                        1\leq D_{r_0}^+(\Gamma)<\infty.
\]
For every $\varepsilon>0$, there is $H_\varepsilon$ such that
$M_\Gamma(h)/h^{r_0}\leq D_{r_0}^+(\Gamma)+\varepsilon$
for $h\geq H_\varepsilon$.
\end{lem}
\begin{proof}
For $h\geq1$ choose the integer $n_h\geq0$ with
$q^{n_h}\leq2h<q^{n_h+1}$. Lemma \ref{thm4.1} for the full word set
gives $M_\Gamma(h)=M_{\Gamma_{n_h+2}}(h)\leq m^{n_h+2}$.
Since $q^{r_0}=m$, we obtain the pointwise estimate
\[
 M_\Gamma(h)\leq m^2q^{n_hr_0}
                    \leq 2^{r_0}m^2h^{r_0}.
\]
This proves the upper bound with $C=2^{r_0}m^2$.
For the lower bound note that all $m^n$ elements of $\Gamma_n$
belong to $B(0,q^n)$, because their values are nonnegative and
strictly less than $q^n$. Hence
\[
 \frac{M_\Gamma(q^n)}{(q^n)^{r_0}}
 \geq\frac{\#\Gamma_n}{q^{nr_0}}=\frac{m^n}{m^n}=1.
\]
Taking limsup along $h=q^n\to\infty$ proves
$D_{r_0}^+(\Gamma)\geq1$.

Finally put $D_*=D_{r_0}^+(\Gamma)<\infty$.
If the last assertion were false for some $\varepsilon>0$, one could
choose $h_j\geq j$ with
$M_\Gamma(h_j)/h_j^{r_0}>D_*+\varepsilon$. The limsup along this
sequence, and hence the full limsup, would be at least
$D_*+\varepsilon$, a contradiction.
\end{proof}

\begin{lem}\label{thm4.3}
Let $\beta,\gamma$ be integers with $0\leq\beta<\gamma$, and let
$S\subset\Gamma_\beta$ with $P=\#S$.
Put
\[
 F=S+q^\beta\Gamma_{\gamma-\beta},\qquad\theta=P/m^\beta.
\]
Then every interval $U$ satisfies
\begin{equation}\tag{4.2}\label{(4.2)}
 \left|\#(F\cap U)-\theta\#(\Gamma_\gamma\cap U)\right|\leq2P.
\end{equation}
Consequently $|M_F(h)-\theta M_{\Gamma_\gamma}(h)|\leq2P$
for every $h>0$.
\end{lem}
\begin{proof}
We group the digit words by their high-digit suffixes. Every word
$J$ of length $\gamma$ has a unique decomposition $J=IV$, with
$I\in\Omega^\beta$ and
$V\in\Omega^{\gamma-\beta}$, and
\[
 T_\gamma(IV)=T_\beta(I)+q^\beta T_{\gamma-\beta}(V).
\]
Thus each fixed suffix corresponds to a block
$B_v=\Gamma_\beta+q^\beta v$, $v\in\Gamma_{\gamma-\beta}$.
The selected points of that block form $S+q^\beta v$ and number $P$.
Because $\Gamma_\beta\subset[0,q^\beta)$ and distinct integer $v$
differ by at least one, the convex hulls of these blocks are disjoint
and ordered in the same way as $v$.

Fix an interval $U$. Let $\mathcal V$ be the set of suffix values
whose full blocks meet $U$, and write $K=\#\mathcal V$.
If $K\geq2$, only the first and last of these blocks can be partly
included: every point of any intermediate block lies between two
points of $U$, and hence belongs to $U$. If $K=1$, there is only one
possibly partial block. For $K=0$ both counts are zero.

We now compare the two counts block by block. On a full block their difference is
$P-\theta m^\beta=0$. On a partial block, both
\[
 \#((S+q^\beta v)\cap U)
 \quad\text{and}\quad\theta\#(B_v\cap U)
\]
lie between zero and $P$, so the absolute value of their difference
is at most $P$. Summing over at most two partial blocks proves
\eqref{(4.2)}. Equivalently,
\[
 \theta\#(\Gamma_\gamma\cap U)-2P
 \leq\#(F\cap U)
 \leq\theta\#(\Gamma_\gamma\cap U)+2P.
\]

For every $x\in\mathbb R$, \eqref{(4.2)} gives
\[
\#(F\cap B(x,h))
\leq \theta M_{\Gamma_\gamma}(h)+2P.
\]
Taking the supremum over $x$, we obtain
\[
M_F(h)\leq \theta M_{\Gamma_\gamma}(h)+2P.
\]
Conversely, \eqref{(4.2)} also gives, for every $x\in\mathbb R$,
\[
\theta\#(\Gamma_\gamma\cap B(x,h))
\leq M_F(h)+2P.
\]
Since $\theta\geq0$, taking the supremum over $x$ yields
\[
\theta M_{\Gamma_\gamma}(h)\leq M_F(h)+2P.
\]
Combining these inequalities proves the lemma.
\end{proof}

A finite subset of $\Gamma_g$ is called \emph{closed under digit
truncation} if deleting any number of its most significant digits
keeps the resulting number in the set. 

\begin{lem}[Finite-stage calibration]\label{thm4.4}
Fix $0<s<r_0$, an integer $g\geq2$, and a finite set
$F_-\subset\Gamma_g$ containing zero and closed under digit truncation.
Put $N=\#F_-$ and $H=q^g$. Given $\tau>0$ and $e>0$, arbitrarily large
integers $\gamma$ admit integers $\alpha<\beta<\gamma$, an integer
$A\geq2$, and a set $S\subset\Gamma_\beta$ with $\#S=P=AN$ such that,
for $F=S+q^\beta\Gamma_{\gamma-\beta}$,
\begin{enumerate}[(i)]
\item $\alpha>g+3$, $F_-\subset F$, $F$ is closed under digit
truncation, and retaining only the first $\alpha$ digits of any
element of $F$ gives an element of $F_-$;
\item with $z=\sup_{h\geq H}M_{\Gamma_\gamma}(h)/h^s$ and
$\theta=P/m^\beta$,
$\tau-e/2\leq\theta z\leq\tau$;
\item $\#F/q^{s\gamma}\leq\tau$, and
\[
 \tau-e\leq\sup_{H\leq h\leq q^\gamma}\frac{M_F(h)}{h^s}
                  \leq\max\{N/H^s,\tau\}+e.
\]
\end{enumerate}
Moreover, $\alpha$ can exceed any prescribed finite bound, and
$\beta\geq c\gamma$ and $\beta-\alpha\asymp\log\beta$ for sufficiently
large $\gamma$, where one may take $c=(r_0-s)/(2r_0)>0$.
\end{lem}
\begin{proof}
We first choose a real parameter and then round it to obtain an
integer number of selected prefixes.
The parameter $\beta$ determines the proportion of selected prefixes,
and $\alpha$ determines the length of the intervening zero block.
The digits in positions $\beta+1,\ldots,\gamma$ are unrestricted. Throughout this proof $g,N,\tau,e$ are fixed and only
$\gamma$ is allowed to tend to infinity. Constants in this proof may
depend on these fixed data.

\smallskip\noindent\textbf{Step 1: estimate the full-block counting quotient.}
Since $\Gamma_\gamma\subset B(0,q^\gamma)$ and
$\#\Gamma_\gamma=m^\gamma$, we have
\[
 z\geq\frac{M_{\Gamma_\gamma}(q^\gamma)}{q^{s\gamma}}
          =q^{(r_0-s)\gamma}.
\]
For $H\leq h\leq q^\gamma$, Lemma \ref{thm4.2} gives
\[
 \frac{M_{\Gamma_\gamma}(h)}{h^s}
 \leq\frac{M_\Gamma(h)}{h^s}
 \leq Ch^{r_0-s}\leq Cq^{(r_0-s)\gamma}.
\]
For $h\geq q^\gamma$ all points fit inside an interval of radius $h$,
so the quotient is $m^\gamma/h^s$ and is decreasing. Therefore
\[
 q^{(r_0-s)\gamma}\leq z\leq Cq^{(r_0-s)\gamma},\qquad
 z=\sup_{H\leq h\leq q^\gamma}\frac{M_{\Gamma_\gamma}(h)}{h^s}.
\]
In particular $\lim_{\gamma\rightarrow\infty}z=\infty$, since $r_0-s>0$.

\smallskip\noindent\textbf{Step 2: choose a real parameter, then round it.}
For this fixed $\gamma$ consider the continuous function
\[
                 G_\gamma(x)=\frac{(x+1)zN}{m^x}.
\]
At $x=g+2$ its value tends to infinity as $\gamma\to\infty$.
At the other endpoint,
\[
 G_\gamma(\gamma)
 \leq CN(\gamma+1)q^{(r_0-s)\gamma-r_0\gamma}
 =CN(\gamma+1)q^{-s\gamma}\longrightarrow0.
\]
For all sufficiently large $\gamma$, these values lie respectively
above and below $\tau$. The intermediate value theorem gives
$b\in(g+2,\gamma)$ with
\[
                  \frac{(b+1)zN}{m^b}=\tau.
\]
Taking logarithms to base $q$ yields the exact identity
\[
 r_0 b=\log_q z+\log_q(b+1)+\log_q(N/\tau).
\]
Step 1 gives $\log_q z=(r_0-s)\gamma+O(1)$.
As $g+2<b<\gamma$, the other variable term is $O(\log\gamma)$.
Hence, with $a=(r_0-s)/r_0\in(0,1)$,
\[
                         b=a\gamma+O(\log\gamma).
\]
Here $O(\log\gamma)$ means that the absolute value of the error is
at most a fixed constant times $\log\gamma$ for large $\gamma$.

Put
\[
 \beta=\lfloor b\rfloor,\qquad
 \kappa=\frac{\tau m^\beta}{zN},\qquad A=\lfloor\kappa\rfloor.
\]
The symbols $\lfloor x\rfloor$ and $\lceil x\rceil$ denote the greatest
integer at most $x$ and the least integer at least $x$, respectively.
The defining equation for $b$ gives
$\kappa=(b+1)m^{\beta-b}$. Since $b-1<\beta\leq b$,
\[
 \frac{b+1}{m}\leq\kappa\leq b+1,
 \qquad \frac{b+1}{m}-1\leq A\leq b+1.
\]
In particular $A\geq2$ for large $\gamma$.
Using $0\leq\kappa-A<1$, we obtain
\[
\begin{split}
 0\leq\tau-\frac{ANz}{m^\beta}
 &=(\kappa-A)\frac{Nz}{m^\beta}
 <\frac{Nz}{m^\beta}=\frac{\tau m^{b-\beta}}{b+1}
 \leq\frac{m\tau}{b+1}\longrightarrow0.
\end{split}
\]
Increase $\gamma$ until this error is below $e/2$.
This proves (ii) once the set of $AN$ prefixes is constructed.
The threshold for $\gamma$ may depend on $\tau$ and $e$, which are
fixed throughout this construction.

\smallskip\noindent\textbf{Step 3: construct the words with the exact count.}
Set $d=\lceil\log_m A\rceil$ and $\alpha=\beta-d$.
Then $m^d\geq A$, so $\Gamma_d$ contains at least $A$ elements.
Let $Q_A$ consist of its $A$ smallest elements and define
\[
 S=F_-+q^\alpha Q_A,
 \qquad F=S+q^\beta\Gamma_{\gamma-\beta}.
\]
The estimates in Step 2 imply
\[
 \beta=a\gamma+O(\log\gamma),\qquad
 d=\log_m\gamma+O(1),\qquad
 \alpha=a\gamma+O(\log\gamma).
\]
Thus $\alpha>g+3$ and $\alpha<\beta<\gamma$ for large $\gamma$;
$\alpha$ can exceed any fixed lower bound. Also
$\beta\geq(a/2)\gamma=c\gamma$ and $d\asymp\log\beta$.
The latter notation means that $d/\log\beta$ stays between two
positive constants for all sufficiently large $\gamma$.

The digit blocks of an element of $F$ occupy disjoint positions:
the first $g$ positions contain a word from $F_-$; positions
$g+1,\ldots,\alpha$ are zero; the next $d$ positions give a word
from $Q_A$; and positions $\beta+1,\ldots,\gamma$ are arbitrary.
Base-$q$ expansion is unique, so these sums have no repeated elements.
In particular
\[
             \#S=AN=P,\qquad\#F=ANm^{\gamma-\beta}.
\]
Both $Q_A$ and the free block contain zero, so $F_-\subset F$.
Deleting high digits from an element of $Q_A$ decreases its value
and keeps it in $\Gamma_d$, hence in the first $A$ elements $Q_A$.
Together with the corresponding property of $F_-$, this proves
closure under digit truncation for $F$: a cut in the free block
leaves a free word, a cut in the selected block leaves a word in
$Q_A$, and a cut in the lowest $g$ positions leaves an element of
$F_-$. A cut in the intervening zero block simply leaves the old
word. All assertions in (i) follow.

\smallskip\noindent\textbf{Step 4: control the endpoint error.}
Set
\[
                     R=\frac{q^\alpha}{4q},\qquad
                     R'=\frac{q^\beta}{4q}.
\]
For large $\gamma$, $H<R<R'<q^\gamma$.
Since $A\leq b+1\leq\gamma+1$, we have $P\leq N(\gamma+1)$.
Also $\alpha/\gamma\to a>0$, so eventually $\alpha\geq a\gamma/2$.
Consequently
\[
 \frac{2P}{R^s}
 \leq2N(\gamma+1)(4q)^sq^{-sa\gamma/2}\longrightarrow0.
\]
Choose $\gamma$ large enough that $2P/R^s<e/4$.

\smallskip\noindent\textbf{Step 5: upper bounds in three ranges of radii.}
If $H\leq h<R$, two points coming from different translated copies
of $F_-$ must differ in some digit above $\alpha$. The highest such
digit and the separation estimate in Lemma \ref{thm4.1} give distance
at least $q^\alpha/(q-1)>2R>2h$. Thus an interval of radius $h$
meets only one copy of $F_-$, and
\[
                         \frac{M_F(h)}{h^s}\leq\frac{N}{H^s}.
\]
If $R\leq h<R'$, the same argument above position $\beta$ shows
that an interval meets at most one translated copy of $S$.
Therefore
\[
                         \frac{M_F(h)}{h^s}\leq\frac{P}{R^s}<e/8.
\]
If $R'\leq h\leq q^\gamma$, Lemma \ref{thm4.3} and Steps 1--2 give
\[
\begin{split}
 \frac{M_F(h)}{h^s}
 &\leq\theta\frac{M_{\Gamma_\gamma}(h)}{h^s}+\frac{2P}{h^s}\leq\theta z+\frac{2P}{R^s}\leq\tau+e/4.
\end{split}
\]
Combining the three estimates proves the upper bound in (iii).

\smallskip\noindent\textbf{Step 6: find a radius giving the lower bound.}
We must first show that the supremum defining $z$ can be taken
over $[R,q^\gamma]$. Lemma \ref{thm4.2} and Step 1 imply
\[
 \frac{1}{z}\sup_{H\leq h<R}
       \frac{M_{\Gamma_\gamma}(h)}{h^s}
 \leq C(4q)^{-(r_0-s)}q^{-(r_0-s)(\gamma-\alpha)}
 \longrightarrow0.
\]
Indeed, Step 3 gives $\alpha/\gamma\to a=(r_0-s)/r_0<1$,
so $\gamma-\alpha\to\infty$. For sufficiently large $\gamma$,
the supremum over $[H,R)$ is therefore strictly smaller than $z$.
Since Step 1 identifies $z$ with the supremum over $[H,q^\gamma]$,
we conclude that
\[
 z=\sup_{R\leq h\leq q^\gamma}
                 \frac{M_{\Gamma_\gamma}(h)}{h^s}.
\]
Now the definition of supremum and $\theta=P/m^\beta>0$ yield
$h_*\in[R,q^\gamma]$ such that
\[
 \theta\frac{M_{\Gamma_\gamma}(h_*)}{h_*^s}>\theta z-e/4.
\]
This choice uses only the definition of supremum and does not require
that it be attained.
The lower inequality from Lemma \ref{thm4.3} now gives
\[
\begin{split}
 \frac{M_F(h_*)}{h_*^s}
 &\geq\theta\frac{M_{\Gamma_\gamma}(h_*)}{h_*^s}
                               -\frac{2P}{h_*^s}>\theta z-e/4-e/4
 \geq\tau-e.
\end{split}
\]
This proves the required lower bound.

Finally the exact count in Step 3 gives
\[
 \frac{\#F}{q^{s\gamma}}
 =\frac{Pm^{\gamma-\beta}}{q^{s\gamma}}
 =\theta q^{(r_0-s)\gamma}\leq\theta z\leq\tau.
\]
All the requirements hold simultaneously for every sufficiently
large $\gamma$, completing the proof.
\end{proof}


\begin{lem}\label{thm4.6}
Let $s>0$. Suppose, for all sufficiently large $h$,
$M_\Lambda(h)/h^s\leq C\log h$, and suppose
\[
       \limsup_{h\to\infty}(\log h)\frac{M_\Lambda(h)}{h^s}>0.
\]
Then $\dim_{\mathrm{Be}}\Lambda=s$.
\end{lem}
\begin{proof}
We prove the upper and lower dimension bounds separately.
For any $u>s$, the assumed upper estimate gives, for all large $h$,
\[
 0\leq\frac{M_\Lambda(h)}{h^u}
   =\frac{M_\Lambda(h)}{h^s}h^{s-u}
   \leq C\frac{\log h}{h^{u-s}}\longrightarrow0.
\]
Thus $D_u^+(\Lambda)=0$ for every $u>s$ and
$\dim_{\mathrm{Be}}\Lambda\leq s$.

For the reverse inequality, a positive limsup means that there are
$c>0$ and a sequence $h_j\to\infty$ satisfying
\[
                  \frac{M_\Lambda(h_j)}{h_j^s}\geq\frac{c}{\log h_j}.
\]
If the limsup is finite, take $c$ smaller than it; if infinite, any
fixed positive $c$ suffices. 
For $0<u<s$ it follows that
\[
 \frac{M_\Lambda(h_j)}{h_j^u}
 =\frac{M_\Lambda(h_j)}{h_j^s}h_j^{s-u}
 \geq\frac{c h_j^{s-u}}{\log h_j}\longrightarrow\infty.
\]
The limsup over all radii is at least the limsup over this sequence,
so $D_u^+(\Lambda)=\infty$ for every $0<u<s$.
Hence $\dim_{\mathrm{Be}}\Lambda\geq s$, completing the proof.
\end{proof}

\begin{thm}\label{thm4.7}
For every $0<s<r_0$ and $\zeta\in[0,+\infty]$, there is a prefix-closed
increasing word sequence $(\mathcal J_n)$ such that
\[
 \dim_{\mathrm{Be}}\Lambda(\mathcal J)=s,
                 \qquad D_s^+(\Lambda(\mathcal J))=\zeta.
\]
It may be constructed with integers
$\gamma_{k-1}<\alpha_k<\beta_k<\gamma_k$ satisfying
$\gamma_{k-1}=o(\gamma_k)$, $\beta_k\geq c\gamma_k$,
$\beta_k-\alpha_k\asymp\log\beta_k$, and
\[
 \mathcal J_{\alpha_k}
       =\mathcal J_{\gamma_{k-1}}0^{\alpha_k-\gamma_{k-1}},
 \qquad
 \mathcal J_{\gamma_k}
       =\mathcal J_{\beta_k}\Omega^{\gamma_k-\beta_k}.
\]
At stage $k$, $\alpha_k$ can exceed any finite lower bound specified
after the preceding stage. More precisely, put
\[
 \tau_k=\begin{cases}
 \zeta,&0<\zeta<\infty,\\
 1/\log_m(k+1),&\zeta=0,\\
 \log_m(k+1),&\zeta=\infty,
 \end{cases}
 \qquad e_k=\frac{\min\{1,\tau_k\}}{(k+1)^2},\qquad H_k=q^{\gamma_k}.
\]
For $k\geq2$, the sequence can be chosen so that
\begin{equation}\tag{4.3}\label{(4.3)}
 \tau_k-e_k\leq
 \sup_{H_{k-1}\leq h\leq H_k}
       \frac{M_{\Lambda(\mathcal J)}(h)}{h^s}
                \leq\max\{\tau_{k-1},\tau_k\}+e_k.
\end{equation}
\end{thm}
\begin{proof}
We construct the word sets stage by stage, then estimate their
interval counts at all large radii. 

\smallskip\noindent\textbf{Step 1: initialization and induction.}
Choose $\gamma_0\geq2$, put $F_0=\{0\}$, and take
$\mathcal J_n=\{0^n\}$ for $0\leq n\leq\gamma_0$;
$0^0$ denotes the empty word.
Suppose the finite set $F_{k-1}\subset\Gamma_{\gamma_{k-1}}$ and
all word sets through length $\gamma_{k-1}$ have been fixed.
Apply Lemma \ref{thm4.4} with
\[
 g=\gamma_{k-1},\quad F_-=F_{k-1},\quad
 N_k=\#F_{k-1},\quad\tau=\tau_k,\quad e=e_k.
\]
Its hypotheses hold initially and are preserved by part (i) of that
lemma. Choose $\gamma_k$ sufficiently large that in addition
\[
 \gamma_k\geq k^2,\qquad
 \gamma_k\geq(k+1)\gamma_{k-1},\qquad
 \alpha_k>\gamma_{k-1}+3,
\]
and any specified finite lower bound on $\alpha_k$ is satisfied.
Since the asymptotic estimates in the lemma are statements for each
fixed stage, enlarge $\gamma_k$ also until
$\beta_k\geq c\gamma_k$ and
$\tfrac12\log_m\beta_k\leq\beta_k-\alpha_k\leq2\log_m\beta_k$. 
Indeed, for each fixed stage, the proof of Lemma \ref{thm4.4}
gives
\[
\beta=a\gamma+O(\log\gamma),\qquad
\beta-\alpha=\log_m\gamma+O(1),
\qquad a=\frac{r_0-s}{r_0}>0.
\]
Consequently,
$(\beta-\alpha)/\log_m\beta\to1$ as $\gamma\to\infty$,
which justifies these bounds.
This gives constants independent of $k$ in the asserted comparison.
Let $d_k=\beta_k-\alpha_k$, and obtain
\[
 S_k=F_{k-1}+q^{\alpha_k}Q_{A_k},\qquad
 F_k=S_k+q^{\beta_k}\Gamma_{\gamma_k-\beta_k}.
\]
Their sizes are $\#S_k=A_kN_k$ and
$\#F_k=A_kN_km^{\gamma_k-\beta_k}$. In particular, for
$H_k=q^{\gamma_k}$, the last assertion of Lemma \ref{thm4.4} gives
\begin{equation}\tag{4.4}\label{(4.4)}
                         \frac{\#F_k}{H_k^s}\leq\tau_k.
\end{equation}
The inequality $\gamma_k\geq(k+1)\gamma_{k-1}$ implies
$\gamma_{k-1}/\gamma_k\leq1/(k+1)\to0$, as required.

\smallskip\noindent\textbf{Step 2: define all intermediate word sets.}
Let $\mathcal Q_k\subset\Omega^{d_k}$ be the words corresponding
to $Q_{A_k}$. At the endpoint set explicitly
\[
 \mathcal J_{\gamma_k}
 =\mathcal J_{\gamma_{k-1}}
       0^{\alpha_k-\gamma_{k-1}}
       \mathcal Q_k\Omega^{\gamma_k-\beta_k}.
\]
Concatenation here means taking every choice from each of the
listed sets of words. Its numerical image is exactly $F_k$.
For $\gamma_{k-1}\leq n\leq\gamma_k$, define $\mathcal J_n$ as the
set of first $n$ digits of words in $\mathcal J_{\gamma_k}$.
Part (i) of Lemma \ref{thm4.4} ensures that at the lower endpoint
this gives the already specified $\mathcal J_{\gamma_{k-1}}$.
Thus consecutive stages agree and no earlier definition is changed.

Every allowed prefix is still allowed at a shorter length, so the
family is prefix closed. Closure of $F_k$ under digit truncation
ensures that, if an allowed word is extended by a zero digit, that
extension also occurs. Explicitly, truncate a word of $F_k$ to the
given prefix and pad the resulting element of $F_k$ by zeros.
This proves $\mathcal J_n0\subset\mathcal J_{n+1}$.
The second inclusion in Definition \ref{de3.1} follows from prefix
closure: if $J0\in\mathcal J_{n+1}$, then $J\in\mathcal J_n$.
 Zero belongs at every
length, so this is an increasing word sequence in the stated sense.

All digits between $\gamma_{k-1}+1$ and $\alpha_k$ are zero, and the
last block is free. Therefore
\[
 \mathcal J_{\alpha_k}
       =\mathcal J_{\gamma_{k-1}}0^{\alpha_k-\gamma_{k-1}},\qquad
 \mathcal J_{\gamma_k}=\mathcal J_{\beta_k}\Omega^{\gamma_k-\beta_k}.
\]
Also $F_{k-1}\subset F_k$ and
$\Lambda(\mathcal J)=\bigcup_k F_k$.

\smallskip\noindent\textbf{Step 3: later stages do not alter counts at earlier radii.}
Fix $k$ and $0<h\leq H_k$. Suppose two points of
$B(x,h)\cap\Lambda(\mathcal J)$ have different tails above position
$\gamma_k$. All their digits from $\gamma_k+1$ through
$\alpha_{k+1}$ are zero, so their highest differing position is
at least $\alpha_{k+1}+1$. Lemma \ref{thm4.1}'s separation estimate
would give distance at least
\[
 \frac{q^{\alpha_{k+1}}}{q-1}>2q^{\gamma_k}=2H_k\geq2h,
\]
contradicting membership in the same interval.
Thus all points share the same tail. Subtracting it puts their
prefixes in $F_k$ by prefix closure and keeps them in an interval
of radius $h$. It follows that
$M_{\Lambda(\mathcal J)}(h)\leq M_{F_k}(h)$.
Since $F_k\subset\Lambda(\mathcal J)$, the reverse inequality holds:
\[
                  M_{\Lambda(\mathcal J)}(h)=M_{F_k}(h)
                                      \quad(0<h\leq H_k).
\]
This equality is why estimates for one finite stage remain valid
for the final infinite set.

\smallskip\noindent\textbf{Step 4: upper and lower bounds on one stage of radii.}
For $k\geq2$ and $H_{k-1}\leq h\leq H_k$, Lemma \ref{thm4.4}(iii)
and Step 3 give
\[
 \tau_k-e_k\leq
 \sup_{H_{k-1}\leq h\leq H_k}
           \frac{M_{\Lambda(\mathcal J)}(h)}{h^s}
 \leq\max\{\#F_{k-1}/H_{k-1}^s,\tau_k\}+e_k.
\]
Use \eqref{(4.4)} from stage $k-1$ in the right side to obtain
\eqref{(4.3)}. 

The definition of supremum also supplies, for every $k\geq2$, a
radius $h_k\in[H_{k-1},H_k]$ with
\[
           \frac{M_{\Lambda(\mathcal J)}(h_k)}{h_k^s}
                                      >\tau_k-2e_k.
\]
As $h_k\geq H_{k-1}\to\infty$, these radii can be used to bound the
limsup from below.

\smallskip\noindent\textbf{Step 5: compute the density in each of the three cases.}
If $0<\zeta<\infty$, all targets equal $\zeta$ and $e_k\to0$.
The upper estimate valid on every stage gives $D_s^+\leq\zeta$;
the radii $h_k$ give $D_s^+\geq\zeta$.

If $\zeta=0$, then $\tau_k\downarrow0$.
On the $k$th stage the upper bound is
$\tau_{k-1}+e_k\to0$. Since the stages cover all sufficiently large
radii and counting quotients are nonnegative, $D_s^+=0$.
At the same time the lower estimates remain positive:
\[
 \tau_k-2e_k\geq\tau_k\left(1-\frac{2}{(k+1)^2}\right)
                      \geq\frac1{2\log_m(k+1)}\qquad(k\geq2).
\]
This quantitative lower bound will preserve the dimension.

If $\zeta=\infty$, then $\tau_k=\log_m(k+1)\to\infty$ and $e_k\to0$.
The quotients at $h_k$ tend to infinity, so $D_s^+=\infty$.
We have proved $D_s^+(\Lambda(\mathcal J))=\zeta$ in every case.

\smallskip\noindent\textbf{Step 6: verify that the dimension is always $s$.}
When $\zeta=0$ or $\zeta=\infty$, the density alone does not determine
the dimension. We verify the hypotheses of Lemma \ref{thm4.6}
for all three cases.
For all large $k$, all three choices of target satisfy
$\max\{\tau_{k-1},\tau_k\}+e_k\leq C_1(1+\log(k+1))$,
where $C_1$ may depend on the fixed $\zeta$ when it is finite positive.
For $h\in[H_{k-1},H_k]$ we have
$\log h\geq\gamma_{k-1}\log q\geq(k-1)^2\log q$.
Increasing a constant to cover the initial stages, \eqref{(4.3)}
therefore gives
\begin{equation}\tag{4.5}\label{(4.5)}
                   \frac{M_{\Lambda(\mathcal J)}(h)}{h^s}
                                      \leq C\log h
\end{equation}
for all sufficiently large $h$.
For $\zeta=0$, Step 5 gives
$M_{\Lambda(\mathcal J)}(h_k)/h_k^s\geq1/(2\log_m(k+1))$.
For $0<\zeta<\infty$, this quotient is eventually at least
$\zeta/2$, and for $\zeta=\infty$ it is eventually at least $1$.
Since $1/(2\log_m(k+1))\to0$, the same lower bound thus holds
in all three cases for sufficiently large $k$. Consequently,
\[
 (\log h_k)\frac{M_{\Lambda(\mathcal J)}(h_k)}{h_k^s}
 \geq\frac{(k-1)^2\log q}{2\log_m(k+1)}\longrightarrow\infty.
\]
Thus the positive-limsup condition of Lemma \ref{thm4.6} holds in
all cases. That lemma proves $\dim_{\mathrm{Be}}\Lambda(\mathcal J)=s$.
\end{proof}

\begin{proof}[Proof of Theorem \ref{thm1.1}]
We combine the prescribed word construction with the spectral
extension, keeping the order of the choices explicit.

\smallskip\noindent\textbf{Step 1: normalize the digits and fix the targets.}
Lemma \ref{thm3.1} allows us to assume
$0\in L\subset\{0,\ldots,q-2\}$, without changing $\mu_{q,D}$ or
$\mathcal T$. Fix $0<s<r_0$ and $\zeta\in[0,+\infty]$.
Use the targets $\tau_k$ and errors $e_k$ of Theorem \ref{thm4.7}.
Also choose a recurrent enumeration $(a_k)$ of $\mathcal T_+$.
For example, if $\mathcal T_+=\{t_1,t_2,\ldots\}$, concatenate
$(t_1)$, $(t_1,t_2)$, $(t_1,t_2,t_3)$, and so on; for a finite set
repeat its entire list. Every multiplier then occurs infinitely often.

\smallskip\noindent\textbf{Step 2: perform the two constructions in alternating order.}
Initially take $m_0=0$ and $M_0=0$, and initialize the finite density
core as in Theorem \ref{thm4.7}. Suppose the choices through stage
$k-1$ have been made. At this point the integers $m_{k-1}$ and all
frequencies for the finitely many binary encodings through stage $k-1$
are known. In particular the maximum $M_{k-1}$ defined in
Theorem \ref{thm3.4} is known.
First choose an integer $v_k\geq1$ with
\[
                         a_kq^{-v_k}<\delta_{a_k}/4.
\]
This choice depends only on $a_k$ and $\delta_{a_k}$.

Apply the next stage of Lemma \ref{thm4.4} with the prescribed
$\tau_k,e_k$, requiring additionally $\alpha_k>m_{k-1}+3$ and
\[
                  a_kq^{-(\gamma_k+v_k)}M_{k-1}<\delta_{a_k}/4.
\]
The first requirement is possible because $\alpha_k$ can exceed
any specified finite bound; the second holds for sufficiently large
$\gamma_k$ because $M_{k-1}$ is already fixed. All other density
requirements are those in Theorem \ref{thm4.7}.
Put $m_k=\gamma_k+v_k$. After the density core is fixed through
$\gamma_k$, extend it by zero digits through $m_k$.
In the next density stage impose $\alpha_{k+1}>m_k+3$, ensuring that
these zero extensions remain exactly the allowed words there.
Thus
\[
 \mathcal J_{m_k}=\mathcal J_{\gamma_k}0^{m_k-\gamma_k},\qquad
 a_kq^{\gamma_k-m_k}<\delta_{a_k}/4,\qquad
 a_kq^{-m_k}M_{k-1}<\delta_{a_k}/4.
\]
These are precisely the size conditions for spectral stage $k$.
Carry out that stage for all its finite binary encodings using
Theorem \ref{thm3.4}. Its successive choices are finite in number,
so they produce a finite maximum $M_k$. Only now proceed to the
next density stage. This proves inductively that all stages can
be completed; no parameter is chosen using a future unknown maximum.

\smallskip\noindent\textbf{Step 3: obtain common spectra with the prescribed core.}
Theorem \ref{thm4.7} gives
\[
 \dim_{\mathrm{Be}}\Lambda(\mathcal J)=s,
 \qquad D_s^+(\Lambda(\mathcal J))=\zeta.
\]
Theorem \ref{thm3.4} supplies, for every $I\in\{1,2\}^{\mathbb N}$,
\[
 \Lambda(I)=\Lambda(\mathcal J)\cup E(I),\qquad
 t\Lambda(I)\text{ is a spectrum of }\mu\quad(t\in\mathcal T),
\]
where $E(I)$ is positive and $q$-lacunary. In the recursion, a core
word has carry zero and a core parent; induction gives
$\lambda_k(J,I|_k)=T_{m_k}(J)$ for every $J\in\mathcal J_{m_k}$.
Thus the same numerical core occurs for every binary encoding, as required.

\smallskip\noindent\textbf{Step 4: preserve the upper Beurling densities.}
For any $h>0$, inclusion and the union bound give
\[
 M_{\Lambda(\mathcal J)}(h)\leq M_{\Lambda(I)}(h)
 \leq M_{\Lambda(\mathcal J)}(h)+M_{E(I)}(h).
\]
Fix $u>0$ and divide by $h^u$. Lemma \ref{thm3.5} implies
$M_{E(I)}(h)/h^u\to0$, so the two other quotients have the same
limsup. If the core limsup is infinite, the lower inequality alone
already gives infinity for the union. Hence in the extended range
$[0,+\infty]$ we have
\[
                    D_u^+(\Lambda(I))=D_u^+(\Lambda(\mathcal J))
                                         \quad(u>0).
\]
Taking $u=s$ proves the required density, including zero and
infinity. Taking the infimum of the exponents for which this density
is zero proves the required Beurling dimension $s$.

\smallskip\noindent\textbf{Step 5: count the distinct spectra.}
Theorem \ref{thm3.4}(iii) proves that different infinite binary
sequences give different sets $\Lambda(I)$. The set of such sequences
has cardinality $2^{\aleph_0}$, so the required family has at least
that cardinality. For the reverse bound, every countable subset of
$\mathbb R$ can be listed in a sequence, and
\[
                 \#\mathbb R^{\mathbb N}
                 =(2^{\aleph_0})^{\aleph_0}=2^{\aleph_0}.
\]
There are therefore at most continuum many countable frequency sets,
and equality follows. All constructed frequencies are nonnegative
integers and zero is retained at every stage. This completes the proof.
\end{proof}

\begin{re}\label{re4.7}
The family has continuum cardinality, although every individual
spectrum is countable. 
\end{re}


\end{document}